\documentclass[a4paper,UKenglish,cleveref, autoref, thm-restate]{lipics-v2021}

\usepackage{amsmath}

\usepackage{amsthm}
\usepackage{amssymb}
\usepackage{amsfonts}
\usepackage{mathrsfs}
\usepackage{wrapfig} 

\usepackage{thmtools} 
\usepackage{thm-restate}

\usepackage{caption} 
\usepackage{changepage} 
\usepackage{mathtools} 

\usepackage{graphicx}
\graphicspath{{Figures/}}

\usepackage{tikz}

\usetikzlibrary{arrows,decorations.pathmorphing,decorations.pathreplacing,backgrounds,positioning,fit,matrix}
\usetikzlibrary{shapes,calc,patterns,arrows.meta}
\tikzset{
	vert/.style={circle,inner sep=1.5,fill=white,draw,minimum size=.3cm},
	edge/.style={color=black, thick},
	diredge/.style={->,>={Stealth[width=8pt,length=8pt]},color=black, thick},
	timelabel/.style={fill=white,font=\footnotesize, text centered},
	wave/.style={decorate,decoration={coil,aspect=0}},
	dirwave/.style={->, >={Stealth[width=8pt,length=8pt]},decorate,decoration={coil,aspect=0}},
	diredge2/.style={->,>={Stealth[width=8pt,length=8pt]}}
}

\usepackage{enumerate}
\usepackage{todonotes}
\usepackage{comment}

\usepackage{algorithm}
\usepackage[noend]{algpseudocode}

\usepackage[T1]{fontenc}
\usepackage[utf8]{inputenc}
\usepackage{lmodern}

\usepackage{hyperref}

\crefname{claim}{Claim}{Claims}

\usepackage{tabularx}

\usepackage{cite} 
\title{On the existence of $\delta$-temporal cliques in random simple temporal graphs} 

\author{George B. Mertzios}{Department of Computer Science, Durham University, UK}{george.mertzios@durham.ac.uk}{https://orcid.org/0000-0001-7182-585X}{Supported by the EPSRC grant EP/P020372/1.}

\author{Sotiris Nikoletseas}{Computer Engineering and Informatics Department, University of Patras, Greece}{nikole@cti.gr}{https://orcid.org/0000-0003-3765-5636}{}

\author{Christoforos Raptopoulos}{Department of Mathematics, University of Patras, Greece}{raptopox@upatras.gr}{https://orcid.org/0000-0002-9837-2632}{Supported by the Hellenic Foundation
for Research and Innovation (H.F.R.I.) under the “2nd Call for H.F.R.I.
Research Projects to support Post-Doctoral Researchers” (Project Number:
704).}

\author{Paul G. Spirakis}{Department of Computer Science, University of Liverpool, UK}{p.spirakis@liverpool.ac.uk}{https://orcid.org/0000-0001-5396-3749}{Supported by the EPSRC grant EP/P02002X/1.}

\authorrunning{George B. Mertzios, Sotiris Nikoletseas, Christofors Raptopoulos, and Paul G. Spirakis} 

\Copyright{George B. Mertzios, Sotiris Nikoletseas, Christofors Raptopoulos, and Paul G. Spirakis} 



\category{} 

\relatedversion{} 

\nolinenumbers 

\EventEditors{John Q. Open and Joan R. Access}
\EventNoEds{2}
\EventLongTitle{42nd Conference on Very Important Topics (CVIT 2016)}
\EventShortTitle{CVIT 2016}
\EventAcronym{CVIT}
\EventYear{2016}
\EventDate{December 24--27, 2016}
\EventLocation{Little Whinging, United Kingdom}
\EventLogo{}
\SeriesVolume{42}
\ArticleNo{23}

\begin{document}
\maketitle

\begin{abstract}
We consider random simple temporal graphs in which every edge of the complete graph $K_n$ appears once within the time interval $[0,1]$ independently and uniformly at random. Our main result is a sharp threshold on the size of  any maximum $\delta$-clique (namely a clique with edges appearing at most $\delta$ apart within $[0,1]$) in random instances of this model, for any constant~$\delta$. In particular, using the probabilistic method, we prove that the size of a maximum $\delta$-clique is approximately $\frac{2\log{n}}{\log{\frac{1}{\delta}}}$ with high probability (whp).~What seems surprising is that, 
even though the random simple temporal graph contains $\Theta(n^2)$ overlapping $\delta$-windows, which (when viewed separately) correspond to different random instances of the Erd\H{o}s-R\'{e}nyi random graphs model, the size of the maximum $\delta$-clique in the former model and the maximum clique size of the latter are approximately the same. Furthermore, we show that the minimum interval containing a $\delta$-clique is $\delta-o(\delta)$ whp. We use this result to show that any polynomial time algorithm for \textsc{$\delta$-Temporal Clique} is unlikely to have very large probability of success.  
%
\end{abstract}


\section{Introduction}\label{intro-sec}

Dynamic network analysis, i.e.~analysis of networks that change over time, is currently one of the most active topics of research in network science and theory. 
Many modern real-life networks are dynamic in nature, in the sense that the network structure 
undergoes discrete changes over time \cite{KlobasMS23,michailCACM,Santoro11}. 
Here we deal with the discrete-time dynamicity of the network links (edges) over a fixed set of nodes (vertices), according to which edges appear in discrete times and are absent otherwise. 
This concept of dynamic network evolution is given by \emph{temporal graphs}~\cite{kempe00,mertziosMCS19}, 
which are also known by other names such as \emph{evolving graphs} \cite{LeonardiALLM16, Ferreira-MANETS-04}, or \emph{time-varying graphs} \cite{krizanc1}.

\begin{definition}[Temporal Graph]
\label{temp-graph-def} A \emph{temporal graph} is a pair $\mathcal{G}=(G,\lambda)$,
where $G=(V,E)$ is an underlying (static) graph and $\lambda :E\rightarrow 2^\mathbb{N}$ is a \emph{time-labeling} function which assigns to every
edge of $G$ a discrete-time label. 
Whenever $|\lambda(e)|\leq 1$ for every $e\in E$, $\mathcal{G}$ is called a \emph{simple} temporal graph.
\end{definition}

Our focus is on \textit{simple} temporal graphs (in which edges appear only once), as, due to their conceptual simplicity, they offer a fundamental model for temporal graphs and they prove to be good prototypes for studying temporal computational problems. 
More specifically, we consider simple temporal graphs whose edge labels are chosen \textit{uniformly at random} from a very large set of possible labels (e.g.~the label of each edge is chosen uniformly at random within $[1,N]$ where $N\rightarrow \infty$). 
This can be equivalently modeled by choosing the time labels uniformly at random as real numbers in the interval $[0,1]$, which leads to the following definition.

\begin{definition}[Random Simple Temporal Graph]
\label{random-simple-temp-graph-def} A \emph{random simple temporal graph} is a pair $\mathcal{G}=(G,\lambda)$, 
where $G=(V,E)$ is an underlying (static) graph and $\{\lambda(e): e \in E\}$ is a set of independent random variables uniformly distributed within $[0,1]$.    
\end{definition}

Note that, in~\cref{random-simple-temp-graph-def}, the probability that two edges lave equal labels is zero. 
For every $v\in V$ and every time slot $t$, we denote the \emph{appearance
of vertex} $v$ \emph{at time} $t$ by the pair $(v,t)$. 
For $Q\subseteq V$, the \emph{restricted temporal graph} $(G,\lambda)|_Q$ is the temporal graph $(G[Q],\{\lambda(e):e\in E(G[Q])\}$.

In the seminal paper of Casteigts, Raskin, Renken, and Zamaraev \cite{DBLP:conf/focs/CasteigtsRRZ21}, the authors consider a related (essentially equivalent to ours) model 
of random simple temporal graphs based on random permutation of edges. 
They provide a thorough study of the temporal connectivity of such graphs and 
they provide sharp thresholds for temporal reachability. Their work motivated our
research in this paper.

In many applications of temporal graphs, information can naturally only move along edges in a way that respects the ordering of their timestamps (i.e.~time labels). That is, information can only flow along sequences of edges whose time labels are increasing (or non-decreasing). 
Motivated by this fact, most studies on temporal graphs have focused on  ``path-related'' problems, such as e.g.~temporal analogues of distance, diameter, reachability, exploration, and
centrality~\cite{DBLP:conf/focs/CasteigtsRRZ21,KlobasMMNZ23,HeegerHMMNS21,akrida16,Erlebach0K21,MertziosMS19,michailTSP16,akridaTOCS17,EnrightMMZ21,ZschocheFMN20,CasteigtsHMZ21,KlobasMMS22}. 
In these problems, the most fundamental notion is that of a \emph{temporal path} from a vertex $u$ to a vertex $v$, which is a path from $u$ to $v$ such that the time labels of the time labels of the edges are increasing (or at least non-decreasing) in the direction from $u$ to $v$. 
To complement this direction, several attempts have been recently made to define meaningful ``non-path'' temporal graph problems which appropriately 
model specific applications. 
Some examples include temporal cliques, cluster editing, temporal vertex cover, temporal graph coloring, temporally transitive orientations of temporal graphs~\cite{viardCliqueTCS16,himmel17,BHMMNS18,ChenMSS18,AkridaMSZ20,HammComplexity22,MertziosMZ21,yu2013algorithms,ghosal2015channel,MertziosMRSZ21}.

What is common to most of the path-related problems is that their extension from static to temporal graphs often follows easily and quite naturally from their static counterparts. 
For example, requiring a graph to be (temporally) connected results in requiring the existence of a (temporal) path among each pair of vertices.
In the case of non-path related problems, the exact definition and its application is not so straightforward. 
For example, defining cliques in a temporal graph as the set of vertices that interact at least once in the lifetime of the graph would be a bit counter intuitive, as two vertices may just interact at the first time step and never again. 
To help with this problem, Viard et al.~\cite{viardCliqueTCS16} introduced the idea of the \emph{sliding time window} of some size $\delta$,
where they define a temporal clique as a set of vertices where in all $\delta$ consecutive time steps each pair of vertices interacts at least once.
There is a natural motivation for this problem, namely to be able to find the contact patterns among high-school students.
Following the idea of Viard et al.~\cite{viardCliqueTCS16}, many other problems on temporal graph were defined wiusing sliding time windows.
For an overview of recent works on sliding windows in temporal graphs, see~\cite{KMS-MFCS23}.

In the next definition we introduce the notion of a $\delta$\textit{-temporal clique} in a random simple temporal graph, and the corresponding maximization problem.

\begin{definition}[$\delta$-\textsc{Temporal Clique}]
Let $(G,\lambda)$ be a random simple temporal graph with $n$ vertices, let $\delta\in [0,1]$, and let $Q\subseteq V$ be a subset of vertices such that $G[Q]$ is a clique. 
The restricted temporal graph $(G,\lambda)|_Q$ is a $\delta$-temporal clique, if $|\lambda(e)-\lambda(e')|\leq \delta$, for every two edges $e,e'$ which have both their endpoints in $Q$.
\end{definition}

\vspace{0,1cm} \noindent \fbox{ 
\begin{minipage}{0.96\textwidth}
 \begin{tabular*}{\textwidth}{@{\extracolsep{\fill}}lr} \textsc{$\delta$-Temporal Clique} \ \ & \\ \end{tabular*}
 
  \vspace{1.2mm}
{\bf{Input:}}  A simple temporal graph $(G, \lambda)$.\\
{\bf{Output:}} A $\delta$-temporal clique $Q$ of $(G, \lambda)$ with maximum cardinality $|Q|$.
\end{minipage}} \vspace{0,3cm}

\noindent\textbf{Our contribution.} 
In this paper, we consider simple random temporal graphs where the underlying (static) graph is the complete graph on $n$ vertices, and we provide a sharp threshold on the size of  maximum $\delta$-cliques in random instances of this model, for any constant $\delta$. In particular, using the probabilistic method, we prove that the size of a maximum $\delta$-clique is approximately $\frac{2\log{n}}{\log{\frac{1}{\delta}}}$ whp (Theorem \ref{theorem:deltatemporalcliquesize}).~What seems surprising is that, 
even though the random simple temporal graph contains $\Theta(n^2)$ overlapping $\delta$-windows, which (when viewed separately) correspond to different random instances of the Erd\H{o}s-R\'{e}nyi model ${\cal G}_{n, \delta}$ (in which edges appear independently with probability $\delta$), the size of the maximum $\delta$-clique and the maximum clique size of the latter are approximately the same. Furthermore, we show that the minimum interval containing a $\delta$-clique is $\delta-o(\delta)$ whp (Theorem \ref{theorem:deltacliqueintervalsize}). We use this result to show that any polynomial time algorithm for \textsc{$\delta$-Temporal Clique} is unlikely to have very large probability of success (Theorem \ref{theorem:largesuccessprobability}). Finally, we discuss some open problems related to the average case hardness of \textsc{$\delta$-Temporal Clique} in the general case.

\section{Existence of $\delta$-\textsc{Temporal Clique}}

We begin with a Lemma regarding the joint density function of the minimum and maximum label.

\begin{lemma}
Let $(G, \lambda)$ be a random simple temporal graph, i.e. $G=(V,E)$ is a graph with $m = |E(G)| \geq 2$ edges, and $\{\lambda(e): e \in E\}$ is a set of independent random variables uniformly distributed within $[0,1]$. Let also $X \stackrel{\text{def}}{=} \min\{\lambda(e): e \in E\}$ and $Y \stackrel{\text{def}}{=} \max\{\lambda(e): e \in E\}$. Then the joint density function of $X,Y$ is given by
\begin{equation}
    f_{X,Y}(x,y) = \left\{ \begin{array}{ll}
       m(m-1) (y-x)^{m-2}, & 0\leq x \leq y \leq 1 \\
       0 & otherwise.
    \end{array}\right.
\end{equation}
\end{lemma}
\begin{proof}
For $0\leq x \leq y \leq 1$, we have that $\Pr(X \geq x, Y \leq y) = (y-x)^m$. Therefore, for any $0\leq x \leq y \leq 1$
\begin{equation}
f_{X,Y}(x,y) = - \frac{\partial^2}{\partial x \partial y} \Pr(X \geq x, Y \leq y) = - \frac{\partial^2}{\partial x \partial y}(y-x)^m = m(m-1) (y-x)^{m-2}.
\end{equation}
\end{proof}

Similarly, we can prove the following:
\begin{lemma}
Let $(G, \lambda)$ be a random simple temporal graph, i.e. $G=(V,E)$ is a graph with $m = |E(G)|$, and $\{\lambda(e): e \in E\}$ is a set of independent random variables uniformly distributed within $[0,1]$. Let also $X \stackrel{\text{def}}{=} \min\{\lambda(e): e \in E\}$. Then the density function of $X$ is given by
\begin{equation}
    f_{X}(x) = m (1-x)^{m-1}, 0\leq x \leq 1.
\end{equation}
\end{lemma}

We now prove the following auxiliary Lemma, which gives an exact formula for the probability that a graph $H$ appears as a subgraph within a $\delta$-window. 

\begin{lemma} \label{lem:subgraph-delta-lives}
Let $(G, \lambda)$ be a random simple temporal graph, i.e. $G=(V,E)$ is a graph, and $\{\lambda(e): e \in E\}$ is a set of independent random variables uniformly distributed within $[0,1]$. Let also $H=(V(H),E(H))$ be a (not necessarily induced) subgraph of $G$, i.e. $H=G[V(H)]$, with $h=|E(H)|$ edges. For any $\delta \in [0,1]$, we have 
\begin{equation}
    \Pr\left(|\lambda(e)-\lambda(e')| \leq \delta, \forall e,e'\in E(H) \right) = h \delta^{h-1}(1-\delta)+\delta^h.
\end{equation}
\end{lemma}
\begin{proof}
We assume without loss of generality that $h>0$, and fix $e_0 \in E(H)$. For any $\delta \in [0,1]$, we have
\begin{eqnarray}
&& \Pr\left(|\lambda(e)-\lambda(e')| \leq \delta, \forall e,e'\in E(H), \textrm{ and } \lambda(e_0) < \lambda(e), \forall e \in E(H)\backslash e_0  \right) \nonumber \\
&& \quad = \int_0^1 \Pr\left( \lambda(e) \in (x, \max\{x+\delta,1\}], \forall e \in E(H)\backslash e_0 | \lambda(e_0)=x \right) dx \\
&& \quad = \int_0^{1-\delta} \Pr\left( \lambda(e) \in (x, x+\delta], \forall e \in E(H)\backslash e_0 | \lambda(e_0)=x \right) dx \nonumber \\
&& \qquad + \int_{1-\delta}^1 \Pr\left( \lambda(e) \in (x, 1], \forall e \in E(H)\backslash e_0 | \lambda(e_0)=x \right) dx \\
&& \quad = \int_0^{1-\delta} \delta^{h-1} dx + \int_{1-\delta}^1 (1-x)^{h-1} dx \\
&& \quad = \delta^{h-1}(1-\delta)+\frac{1}{h} \delta^h.
\end{eqnarray}
The proof is completed by taking the union over all $e_0 \in E(H)$. 
\end{proof}

The following Theorem is a direct consequence of Lemma \ref{lem:subgraph-delta-lives} (for $h=\binom{k}{2}$) and linearity of expectation.
\begin{theorem} \label{thm:firstmoment}
Let $(K_n, \lambda)$ be a random simple temporal graph where the underlying graph is the complete graph with $n$ vertices. For any $\delta \in [0,1]$, and $k \in [n]$, the expected number of $\delta$-temporal cliques of size $k$ in $(K_n, \lambda)$ is
\begin{displaymath}
\binom{n}{k} \left(\binom{k}{2} \delta^{\binom{k}{2}-1}(1-\delta)+ \delta^{\binom{k}{2}} \right).    
\end{displaymath}
\end{theorem}

By simple calculations and Theorem \ref{thm:firstmoment}, we get the following:
\begin{lemma}[First moment] \label{lemma:firstmoment}
Let $(K_n, \lambda)$ be a random simple temporal graph where the underlying graph is the complete graph with $n$ vertices, and let $\delta \in (0,1)$ be a constant. For any integer $k \in [n]$, let $X^{(k)}$ denote the number of $\delta$-temporal cliques of size $k$ in $(K_n,\lambda)$. Define $k_0 \stackrel{\text{def}}{=} \frac{2 \log{n}}{\log{\frac{1}{\delta}}}$. For any constant $\epsilon>0$ that can be arbitrarily small, we have
\begin{description}
    \item[(i)] $\mathbb{E}[X^{((1+\epsilon)k_0)}] \to 0$, and
    \item[(ii)] $\mathbb{E}[X^{((1-\epsilon)k_0)}] \to \infty$,
\end{description}
as $n \to \infty$. 
\end{lemma}
\begin{proof}
By Theorem \ref{thm:firstmoment}, for any $\delta \in [0,1)$, the expected number of $\delta$-temporal cliques of size $k$ in $(K_n, \lambda)$ is 

\begin{eqnarray}
\mathbb{E}[X^{(k)}] & = & \binom{n}{k} \left(\binom{k}{2} \delta^{\binom{k}{2}-1}(1-\delta)+ \delta^{\binom{k}{2}} \right) \nonumber \\
& \leq & \binom{n}{k} \binom{k}{2} \delta^{\binom{k}{2}-1} \leq \left( \frac{ne}{k}\right)^k  k^2 \delta^{\binom{k}{2}-1} \nonumber \\
& = & \exp\left\{k \ln{n} - \frac{k^2}{2} \ln{\frac{1}{\delta}} - \Theta(k \ln{k}) \right\}. \label{eq:6}
\end{eqnarray}
In particular, since $k_0 = \Theta(\log{n})$, for any constant $\delta \in [0,1)$, the RHS of (\ref{eq:6}) goes to 0, for any $k \geq (1+\epsilon)k_0$, as $n \to \infty$. This proves part (i). 

For part (ii), we note that similar  calculations leading to (\ref{eq:6}) can also be used to prove a lower bound (except for a different constant hidden in the $\Theta$ term in the exponent). In particular, 
\begin{eqnarray}
\mathbb{E}[X^{(k)}] & \geq & \left( \frac{n}{k}\right)^k  \left(\frac{k}{2}\right)^2 \delta^{\binom{k}{2}-1} (1-\delta) = \exp\left\{k \ln{n} - \frac{k^2}{2} \ln{\frac{1}{\delta}} - \Theta(k \ln{k}) \right\}, \label{eq:7}
\end{eqnarray}
which goes to $\infty$, for any $k \leq (1-\epsilon)k_0$, as $n \to \infty$.
\end{proof}

The following lemma concerns the variance of the number of $\delta$-temporal cliques.

\begin{lemma}[Second moment] \label{lemma:secondmoment}
Let $(K_n, \lambda)$ be a random simple temporal graph where the underlying graph is the complete graph with $n$ vertices, and let $\delta \in (0,1)$ be a constant. For any integer $k \in [n]$, let $X^{(k)}$ denote the number of $\delta$-temporal cliques of size $k$ in $(K_n,\lambda)$. Define $k_0 \stackrel{\text{def}}{=} \frac{2 \log{n}}{\log{\frac{1}{\delta}}}$. For any constant $\epsilon>0$ that can be arbitrarily small, we have
$\frac{\mathbb{E}\left[\left( X^{(k)}\right)^2 \right]}{\mathbb{E}^2\left[X^{(k)} \right]} \to 1$, for any $k \leq (1-\epsilon)k_0$, as $n \to \infty$.
\end{lemma}
\begin{proof}
Set $k=(1-\epsilon)k_0$, and note that, by part (ii) of Lemma \ref{lemma:firstmoment}, we have $\mathbb{E}[X^{(k)}] = \omega(1)$. Let $S_1, S_2, \ldots, S_{\binom{n}{k}}$ be an arbitrary enumeration of all cliques of size $k$ in $K_n$. For any $i \in \{1, 2, \ldots, \binom{n}{k}\}$, denote by $X_i$ the indicator random variable that is equal to 1 if $S_i$ is a $\delta$-temporal clique in $(K_n, \lambda)$ and 0 otherwise. In particular, we have $X^{(k)} = \sum_{i=1}^{\binom{n}{k}} X_i$ and so $\left( X^{(k)}\right)^2 = \sum_{i,j} X_i X_j$. Taking expectations we get
\begin{eqnarray}
\mathbb{E}\left[\left( X^{(k)}\right)^2 \right] & = & \sum_{i,j} \mathbb{E}[X_i X_j] \nonumber \\
& = & \sum_{t = 0}^{k} \sum_{i,j: |S_i \cap S_j|=t} \mathbb{E}[X_i X_j] \nonumber \\
& = & \sum_{t = 0}^{k} \sum_{i,j: |S_i \cap S_j|=t} \mathbb{E}[X_i] \mathbb{E}[X_j| X_i=1] \nonumber \\
& = & \sum_{t = 0}^{k} \sum_{i,j: |S_i \cap S_j|=t} \mathbb{E}[X_i] \mathbb{E}[X_j| X_i=1] \label{eq:secondmoment}
\end{eqnarray}

Observe that, by independence, the terms in the above sum corresponding to $t=0$ are equal to $\sum_{i,j: |S_i \cap S_j|=\emptyset} \mathbb{E}[X_i] \mathbb{E}[X_j] = \binom{n}{k} \binom{n-k}{k} \left( \binom{k}{2} \delta^{\binom{k}{2}-1} (1-\delta) + \delta^{\binom{k}{2}} \right)^2$, where in the last equation we applied Lemma \ref{lem:subgraph-delta-lives} with $h=\binom{k}{2}$. In particular, we have 
\begin{eqnarray}
\frac{\sum_{i,j: |S_i \cap S_j|=\emptyset} \mathbb{E}[X_i] \mathbb{E}[X_j]}{\mathbb{E}^2\left[X^{(k)} \right]} & = & \frac{\binom{n-k}{k}}{\binom{n}{k}} = \frac{(n-k) \cdots (n-2k+1)}{n \cdots (n-k+1)} \nonumber \\
& = & \left( 1 - \frac{k}{n} \right) \cdots \left( 1 - \frac{k}{n-k+1} \right) = 1-o(1), \label{eq:termt=k}
\end{eqnarray}
where the last equation follows from the fact $k= o(n)$. Furthermore, it is easy to see that the terms in the sum of the RHS of (\ref{eq:secondmoment}) corresponding to $t=k$ is
equal to $\mathbb{E}[X^{(k)}] = o(\mathbb{E}^2[X^{(k)}])$, since $k=(1-\epsilon)k_0$ and so $\mathbb{E}[X^{(k)}]$ goes to $\infty$ as $n \to \infty$. Therefore, by (\ref{eq:secondmoment}) and (\ref{eq:termt=k}),
\begin{equation}
\frac{\mathbb{E}\left[\left( X^{(k)}\right)^2 \right]}{\mathbb{E}^2\left[X^{(k)} \right]} = 1+o(1)+ \frac{\sum_{t = 1}^{k-1} \sum_{i,j: |S_i \cap S_j|=t} \mathbb{E}[X_i] \mathbb{E}[X_j| X_i=1]}{\mathbb{E}^2\left[X^{(k)} \right]}. \label{eq:secondmomentfinal}
\end{equation}

In view of the above, in what follows we will show that the sum in the RHS of the above equation is $o(1)$, which will prove the theorem. To this end, notice that $\mathbb{E}[X_j| X_i=1]$ is equal to the probability that $S_j$ is a $\delta$-temporal clique given that $S_i$ is a $\delta$-temporal clique. By independence of label choices, when $|S_i \cap S_j|=t$ (i.e., $S_i$ intersects with $S_j$ on $t$ vertices), $\mathbb{E}[X_j| X_i=1]$ is upper bounded by the probability that all labels of edges that have both endpoints in $S_j$ but not both endpoints in $S_i \cap S_j$ are at most $\delta$ apart. Therefore, by applying Lemma \ref{lem:subgraph-delta-lives} with $h=\binom{k}{2} - \binom{t}{2}$, we get 
\begin{eqnarray} 
\mathbb{E}[X_j| X_i=1] & \leq & \left( \binom{k}{2} - \binom{t}{2} \right) \delta^{\binom{k}{2} - \binom{t}{2}-1} (1-\delta) + \delta^{\binom{k}{2} - \binom{t}{2}} \nonumber \\
& \leq & \binom{k}{2} \delta^{\binom{k}{2} - \binom{t}{2}-1}. \label{ineq:condprobbound}
\end{eqnarray}
Using the bound (\ref{ineq:condprobbound}), and noting that $\mathbb{E}[X_i] = \binom{k}{2} \delta^{\binom{k}{2}-1} (1-\delta) + \delta^{\binom{k}{2}}$, for any $i \in \{1, 2, \ldots, \binom{n}{k}\}$, we get
\begin{eqnarray}
&& \frac{\sum_{t = 1}^{k-1} \sum_{i,j: |S_i \cap S_j|=t} \mathbb{E}[X_i] \mathbb{E}[X_j| X_i=1]}{\mathbb{E}^2\left[X^{(k)} \right]}
\nonumber \\
&& \qquad \leq \frac{\sum_{t = 1}^{k-1} \sum_{i,j: |S_i \cap S_j|=t} \left( \binom{k}{2} \delta^{\binom{k}{2}-1} (1-\delta) + \delta^{\binom{k}{2}} \right) \binom{k}{2} \delta^{\binom{k}{2} - \binom{t}{2}-1}}{\binom{n}{k}^2 \left( \binom{k}{2} \delta^{\binom{k}{2}-1} (1-\delta) + \delta^{\binom{k}{2}} \right)^2}
\nonumber \\
&& \qquad = \frac{\sum_{t = 1}^{k-1} \binom{n}{k} \binom{k}{t} \binom{n-k}{k-t}  
\binom{k}{2} \delta^{\binom{k}{2} - \binom{t}{2}-1} }{\binom{n}{k}^2 \left( \binom{k}{2} \delta^{\binom{k}{2}-1} (1-\delta) + \delta^{\binom{k}{2}} \right)} \nonumber \\
&& \qquad  \leq \sum_{t = 1}^{k-1} \frac{ \binom{k}{t} \binom{n-k}{k-t} \binom{k}{2} \delta^{\binom{k}{2} - \binom{t}{2}-1} }{\binom{n}{k} \binom{k}{2} \delta^{\binom{k}{2}-1} (1-\delta)} = \sum_{t = 1}^{k-1} \frac{ \binom{k}{t} \binom{n-k}{k-t} }{\binom{n}{k} \delta^{\binom{t}{2}} (1-\delta)} \nonumber \\ 
&& \qquad \leq \sum_{t = 1}^{k-1} \frac{ \frac{k^t e^t}{t^t} \frac{n^{k-t} e^{k-t}}{(k-t)^{k-t}} }{ \frac{n^k}{k^k} \delta^{\binom{t}{2}} (1-\delta)} = \sum_{t = 1}^{k-1} \frac{k^{k+t} e^k}{t^t (k-t)^{k-t} n^t \delta^{\binom{t}{2}} (1-\delta)}. \label{ineq:lastbound}
\end{eqnarray}
By straightforward analysis on the function $g(t) = \frac{k^{k-t}}{t^t (k-t)^{k-t}}$, for $t \in [1,k-1]$, we get that it is strictly decreasing and so it is maximized at $t=1$, namely $\max\{g(t), 1 \leq t \leq k-1\}=g(1)= \left( 1 - \frac{1}{k} \right)^{k-1} \leq \frac{1}{e}$. Therefore, by (\ref{ineq:lastbound}),
\begin{eqnarray}
\frac{\sum_{t = 1}^{k-1} \sum_{i,j: |S_i \cap S_j|=t} \mathbb{E}[X_i] \mathbb{E}[X_j| X_i=1]}{\mathbb{E}^2\left[X^{(k)} \right]}
& \leq & \sum_{t = 1}^{k-1} \frac{1}{e(1-\delta)} \frac{k^{2t} e^k}{n^t \delta^{\binom{t}{2}}} \nonumber \\
& = & \sum_{t = 1}^{k-1} \frac{1}{e(1-\delta)} e^{2t \ln{k}+k - t \ln{n} + \frac{t (t-1)}{2} \ln{\frac{1}{\delta}}} \nonumber \\
& \leq & \sum_{t = 1}^{k-1} \frac{1}{e(1-\delta)} e^{2t \ln{k}+k - t \ln{n} + t \frac{k \ln{\frac{1}{\delta}}}{2} } \nonumber \\
& = & \sum_{t = 1}^{k-1} \frac{1}{e(1-\delta)} e^{2t \ln{k}+k - \epsilon t \ln{n}} =o(1),
\end{eqnarray}
where in the last equality we used the fact that $k=(1-\epsilon) \frac{2 \ln{n}}{ \ln{\frac{1}{\delta}}}$ and $\ln{k} = o(\ln{n})$. This completes the proof.
\end{proof}

By the probabilistic method, combining Lemma \ref{lemma:firstmoment} and Lemma \ref{lemma:secondmoment}, we get the following:
\begin{theorem} \label{theorem:deltatemporalcliquesize}
Let $(K_n, \lambda)$ be a random simple temporal graph where the underlying graph is the complete graph with $n$ vertices, and let $\delta \in (0,1)$ be a constant. Define $k_0 \stackrel{\text{def}}{=} \frac{2 \log{n}}{\log{\frac{1}{\delta}}}$. As $n\to \infty$ we have the following:
\begin{description}
    \item[(i)] With high probability, $(K_n,\lambda)$ has no $\delta$-temporal clique of size $(1+o(1)) k_0$.
    \item[(ii)] With high probability, $(K_n,\lambda)$ contains a $\delta$-temporal clique of size $(1-o(1)) k_0$.
\end{description}
\end{theorem}

We note that the above Theorem implies the following:

\begin{theorem} \label{theorem:deltacliqueintervalsize}
Let $(K_n, \lambda)$ be a random simple temporal graph where the underlying graph is the complete graph with $n$ vertices, and let $\delta \in (0,1)$ be a constant. Let also $k_0 = \frac{2 \log{n}}{\log{\frac{1}{\delta}}}$ and let $Q$ be any $\delta$-temporal clique of size at least $(1-o(1))k_0$. Define the interval $\Delta(Q) \stackrel{def}{=}\left[ \min(\lambda(e): e \in Q), \max(\lambda(e): e \in Q) \right]$. Then $|\Delta(Q)|=\delta - o(\delta)$ whp.
\end{theorem}
\begin{proof}
Since $|\Delta(Q)| \leq \delta$, by definition of a $\delta$-temporal clique, we only need to prove that $|\Delta(Q)| \geq \delta - o(\delta)$. Suppose there is some $\delta' \leq \delta-\epsilon$, for some positive constant $\epsilon$, such that $\Pr(\exists Q: \Delta(Q) \leq \delta') \geq \epsilon$. This would imply that $Q$ is a $\delta'$-temporal clique in $(K_n, \lambda)$, of size $(1-o(1))k_0$. However, Theorem \ref{theorem:deltatemporalcliquesize}, suggests that the largest $\delta'$-temporal clique in $(K_n, \lambda)$ is at most $(1-o(1)) \frac{2 \log{n}}{\log{\frac{1}{\delta'}}}$, which is much smaller than $(1-o(1))k_0$, leading to a contradiction.   
\end{proof}

\begin{note}
The above proofs also work for smaller values of $\delta=o(1)$ (e.g. $\delta=\frac{1}{\log{\log{n}}}$), but not too small (e.g. if $\delta = O(1/n)$ the expected size of a $\delta$-temporal clique becomes constant and concentration results do not hold). 
\end{note}

\section{Average case hardness implications and open problems}

The threshold given in Theorem \ref{theorem:deltatemporalcliquesize} on the size of the maximum $\delta$-clique reveals an interesting connection between simple random temporal graphs $(K_n, \lambda)$ and Erd\H{o}s-R\'{e}nyi random graphs $G_{n,\delta}$. On one hand, notice that, if we only consider edges with labels within a given $\delta$-window, then the corresponding graph is an instance of ${\cal G}_{n,\delta}$, which has maximum clique size asymptotically equal to $k_0 \stackrel{def}{=} \frac{2 \log{n}}{\log{\frac{1}{\delta}}}$ whp. On the other hand, the random simple temporal graph contains $\Theta(n^2)$ different instances of ${\cal G}_{n, \delta}$, but the size of a maximum $\delta$-clique size is asymptotically the same. One explanation why this happens is that the different instance of ${\cal G}_{n, \delta}$ contained in the random simple temporal graph are highly dependent, even if these correspond to disjoint $\delta$-windows (indeed, edges with labels appearing in one window do not appear in the other and vice versa).  

It is therefore interesting to ask whether we can use the above connection algorithmically. One direction is clearly easier than the other: If there is a polynomial time algorithm ${\cal A}_{ER}(\delta)$ that can find a clique of size $q = \Theta(k_0)$ in a random instance of ${\cal G}_{n,\delta}$ whp, then we can use this algorithm to find an asymptotically equally large $\delta$-clique in a random instance of $(K_n, \lambda)$ with the same probability of success. We note that, finding a clique of size asymptotically close to $k_0$ in $G_{n, \delta}$ is believed to be hard in the average case and there is no known algorithm for this problem that runs in polynomial time in $n$. 

For the other direction, we conjecture that the following reduction may be possible: 

\begin{conjecture} \label{conjecture:reduction}
    Suppose that, for any $\delta \in [0,1]$
    there is a polynomial time algorithm ${\cal A}_{SRT}(\delta)$ that finds an $(1-o(1))$-approximation of a maximum $\delta$-clique in a random instance of $(K_n, \lambda)$ whp. Then ${\cal A}_{SRT}(\delta)$ can be used to design a polynomial time algorithm that finds an $(1-o(1))$-approximation of a maximum in ${\cal G}_{n, \delta}$ whp.
\end{conjecture}

It is clear that the probability of success of ${\cal A}_{SRT}(\delta)$ in the above Conjecture cannot be equal to 1 unless $P=NP$. In the following Theorem we also prove that the probability of success is unlikely to be too large.

\begin{theorem} \label{theorem:largesuccessprobability}
Suppose that, for any constant $\delta \in (0,1)$, the probability of success of algorithm ${\cal A}_{SRT}(\delta)$ is $1-\exp(-\omega(n^2))$. Then ${\cal A}_{SRT}(\delta/2)$ can be used to find a clique of size $(1-o(1))k_0$ in ${\cal G}_{n, \delta}$ whp.
\end{theorem}
\begin{proof}
Let $G_{n,\delta}$ be a random instance of the Erd\H{o}s-R\'{e}nyi random graphs model. 
Let ${\cal E}$ be the event that, for a random instance of the random simple temporal graph $(K_n, \lambda)$, all edges of $G_{n, \delta}$ appear inside $[0,\delta/2]$, while all other edges appear inside $[\delta, 1]$. By definition we then have that
\begin{equation}
\Pr({\cal E}) = \left( \frac{\delta}{2}\right)^{|E(G_{n, \delta})|} (1-\delta)^{\binom{n}{2}-|E(G_{n, \delta})|} \geq \left(\min\{\delta/2, 1-\delta\}\right)^{n^2} \geq \left(\frac{\delta(1-\delta)}{2}\right)^{n^2}.
\end{equation}
Furthermore, notice that, by Theorem \ref{theorem:deltacliqueintervalsize}, given ${\cal E}$, any maximum $\frac{\delta}{2}$-clique belongs to $G_{n,\delta}$ and thus lies within $\left[ 0, \frac{\delta}{2}\right]$ whp. Indeed, a maximum clique of $G_{n,\delta}$ has size asymptotically equal to $k_0$, and there is no $\frac{\delta}{2}$-clique of size $(1-o(1)) k_0$ within $[\delta,1]$, whp. In particular, denoting by ${\cal C}$ the event that the maximum $\frac{\delta}{2}$-clique lies within $[0,\delta/2]$, we have
\begin{equation}
\Pr({\cal E} \cap {\cal C}) = \Pr({\cal E})\Pr({\cal C} | {\cal E}) \geq (1-o(1))\left(\frac{\delta(1-\delta)}{2}\right)^{n^2} = \exp(-\Theta(n^2)). \label{eq:ECbound}   
\end{equation}

In view of the above, given $G_{n, \delta}$, we construct the input instance for ${\cal A}_{SRT}$ as follows: (a) select the labels of edges in $E(G_{n, \delta})$ uniformly at random within $[0,\delta/2]$ and (b) select the labels of all other edges independently, uniformly at random within $[\delta, 1]$. 

Since the failure probability of ${\cal A}_{SRT}(\delta/2)$ is, by assumption, at most $\exp(-\omega(n^2))$ (which is asymptotically smaller than the lower bound given in (\ref{eq:ECbound})),
the algorithm may only fail to find a $\frac{\delta}{2}$-temporal clique of size $(1-o(1)) k_0$ in a vanishing fraction of input instances ${\cal I}$ created. In particular, the $\frac{\delta}{2}$-temporal clique constructed by ${\cal A}_{SRT}(\delta/2)$ will be an $(1-o(1))k_0$ clique of $G_{n, \delta}$ whp, as needed.
\end{proof}

In view of the above proof, one possible approach for Conjecture \ref{conjecture:reduction} would be to examine the distribution of the minimum interval containing all labels of edges in a $\delta$-clique found by ${\cal A}_{SRT}$ (namely 
$\Delta(Q) \stackrel{def}{=}\left[ \min(\lambda(e): e \in Q), \max(\lambda(e): e \in Q) \right]$), when the input instance ${\cal I}= {\cal I}(G_{n, \delta})$ is constructed as follows: Given $G_{n,\delta}$, (1) select labels $\{\lambda(e): e \in E(G_{n,\delta})\}$ independently, uniformly at random (i.u.a.r.) within $[0,\delta]$, and (2) select the rest of the labels $\{\lambda(e): e \notin E(G_{n,\delta})\}$ i.u.a.r. from $[\delta,1]$. In particular, we conjecture the following:

\begin{conjecture} \label{conjecture:uarcliqueinterval}
    Let $Q^*$ be the $\delta$-clique constructed by algorithm ${\cal A}_{SRT}$, on instance ${\cal I}(G_{n,\delta})$. Then $\Delta(Q^*)$ is distributed almost uniformly at random within $[0,1]$.
\end{conjecture}


\end{document}